\documentclass[12pt]{amsart}

\newif\ifdebug
\debugfalse

\newif \iffig
\figfalse

\newif \iftable
\tablefalse

\usepackage{sfilip_package_settings}
\usepackage{sfilip_abbreviations}
\usepackage{sfilip_thm_style_long}
\usepackage{verbatim}

\title[Universal affine bundles for compact complex manifolds]{Universal affine bundles for compact complex manifolds}

\hypersetup{
	pdftitle={Universal affine bundles for compact complex manifolds},	
	pdfauthor={Simion Filip, Valentino Tosatti},	
	pdfsubject={math},		
	pdfkeywords={MSC},	
	pdfnewwindow=true,		
	linkcolor=darkblue,		
	citecolor=darkred,		
	filecolor=darkblue,		
	urlcolor=darkblue,		
	pdfborder={0 0 0},
	breaklinks=true
}

\thanks{Revised \textsc{\today} }

\date{June 2026}
\dedicatory{To Fedya Bogomolov on the occasion of his 80th birthday.}

\author{
	Simion Filip
}
\address{
	\parbox{0.5\textwidth}{
		Department of Mathematics\\
		University of Chicago\\
		5734 S University Ave\\
		Chicago, IL 60637\\}
}
\email{{sfilip@math.uchicago.edu}}

\author{
	Valentino Tosatti
}
\address{
	\parbox{0.7\textwidth}{Courant Institute School of Mathematics, Computing, and Data Science\\
New York University\\
251 Mercer St\\
New York, NY 10012\\}
}
\email{{tosatti@cims.nyu.edu}}
\begin{document}

\begin{abstract}
	We construct a universal affine bundle $E$ over a compact complex manifold $X$ equipped with a probability measure $\mu$, which is affine over the vector space $H_{1,1}(X;\bR)$, and satisfies a number of natural properties.
	The bundle $E$ carries a natural action of the automorphism group of $(X,\mu)$, and provides potentials for all $dd^c$-closed $(1,1)$-forms on $X$.
	We relate our construction to lifts of the action of the automorphism group of a Calabi--Yau manifold to its universal torsor.
\end{abstract}


%
\maketitle
%
\noindent\hrulefill
\tableofcontents
\nointerlineskip
\noindent \hrulefill
\ifdebug
   \listoffixmes
\fi


\section{Introduction}
	\label{sec:introduction}

Associated to any algebraic variety $X$ is a natural torus $\bbT_X$ whose character group is the \Neron--Severi group of $X$.
Assuming $\Pic^0(X)=0$, a $\bbT_X$-torsor $\cU\to X$ is called \emph{universal} if, for any line bundle $L\to X$ there exists an isomorphism of $\bbG_m$-torsors $\cU\times_{[L]}\bbG_m\toisom L^\times$, where $[L]$ represents the character of $\bbT_X$ corresponding to $L$.

Universal torsors have been studied extensively, mainly in the context of Fano varieties and with applications to the (non-)existence and counting of rational points over number fields, see e.g. Colliot-Th\'{e}l\`ene and Sansuc \cite{Colliot-TheleneSansuc1976_Torseurs-sous-des-groupes-de-type-multiplicatif;-applications-a-letude,Colliot-TheleneSansuc1987_La-descente-sur-les-varietes-rationnelles.-II} or Hassett and Tschinkel \cite{HassettTschinkel2004_Universal-torsors-and-Cox-rings,HassettTschinkel2022_Torsors-and-stable-equivariant-birational-geometry}.
In this note, we develop some tools to study the case when $X$ is a Calabi--Yau manifold, equipped with an action of an infinite group $\Gamma$ of automorphisms.

Our main result is \autoref{thm:properties_of_universal_bundle} which constructs, for any compact complex manifold $X$ equipped with a probability measure $\mu$, a smooth bundle $E\to X$ which is affine over the vector space $H_{1,1}(X;\bR)$, and which has a number of natural properties described therein.
A key feature of $E$ is that it carries potentials for all $dd^c$-closed $(1,1)$-forms on $X$.
Additionally, the affine bundles $E$ behave naturally under morphisms of varieties and push-forward of measures, as described in \autoref{thm:functoriality_of_e}.

With this in mind, we expect the following conjecture to hold:

\begin{conjecture*}[Lift in Calabi--Yau case]
	\label{cjc:Lift-in-Calabi-Yau-case}
	Suppose that $X$ is a projective Calabi--Yau manifold (i.e. the canonical bundle is torsion, and assume $H^1(X,\cO_X)=0$), and $\cU$ is a universal torsor for $X$.
	
	The action of $\Aut(X)$ on $X$ lifts to an action on $\cU$, after possibly passing to a finite index subgroup.
	More generally, the same should hold for the group of birational automorphisms of $X$.
\end{conjecture*}

It is clear that if $\Aut(X)$ fixes a point, then the action lifts to the universal torsor.
So in a way, the conjecture says that the volume form on a Calabi--Yau manifold can serve as a substitute for a rational point.

In general, the action of a group $\Gamma$ on $X$ leads to the action of a $\bbT_X$-extension of $\Gamma$ on the universal torsor $\cU$, and the lift of the action is equivalent to the splitting of the short exact sequence:
\[
	1 \to \bbT_X \to \wtilde{\Gamma} \to \Gamma \to 1.
\]
Our main result towards the above Conjecture is \autoref{thm:reduction_to_unitary_part}, which states that if $\Gamma$ is a group of automorphisms of a compact complex manifold $X$ preserving a probability measure, then the above short exact sequence splits after quotienting by the compact subgroup of $\bbT_X(\bC)$.
This would suffice to establish the Conjecture if we knew that the associated cohomology class in $H^2(\Gamma,\bbT_X(\bC))$ is integral.
Note that \autoref{thm:properties_of_universal_bundle} and \autoref{thm:reduction_to_unitary_part} do not require $X$ to be Calabi--Yau, or even simply-connected.

We end with a few remarks and questions:
\begin{remark*}[]
	\label{remark:refinements}
	\leavevmode
	\begin{enumerate}
	\item Constructing a universal torsor $\cU$, or an affine bundle $E$ which provides potentials for all $(1,1)$-forms, can always be done by choosing an explicit basis of the \Neron--Severi group or $H^{1,1}(X)$.
	The difficulty is in extending the group action to the resulting space, and this is where our universal construction in \autoref{thm:properties_of_universal_bundle} provides a natural solution.

	\item Are there examples of varieties $X$ with actions of groups $\Gamma$  on $X$ where it is \emph{not} possible to lift the action to the universal torsor, even after passing to a finite index subgroup in $\Gamma$?
	
	In \autoref{ssec:calabi_yau_example_without_lifted_action} we provide some examples of Calabi--Yau manifolds with actions of finite groups $\Gamma$ which do not lift to the universal torsor, as well as some examples of group actions that do not preserve a measure but nonetheless lift to the universal torsor.

	\item In \cite{FT_canonical}, we constructed canonical currents on the boundary of the ample cone of a K3 surface $X$, which are equivariant under the action of $\Aut(X)$.
	The affine bundle $E$ that we construct in this paper carries potentials for these currents, that behave equivariantly under the action of $\Aut(X)$.

	\item
	It would be of interest to perform the above constructions in the situation when $\Pic^0(X)\neq 0$.
	In this case the universal torsor should be replaced by a universal semi-abelian variety.
	\end{enumerate}
\end{remark*}

\subsubsection*{Acknowledgements}
	\label{sssec:Acknowledgements}
It is our pleasure to dedicate this paper to Fedya Bogomolov on the occasion of his 80th birthday. His seminar work and extensive mathematical breadth are an inspiration for us.

We are also grateful to Robert Berman and Mihai P\u{a}un for helpful discussions regarding the construction of natural measures on \Kahler manifolds with effective canonical bundle.

This material is based upon work partially supported by the National Science Foundation under Grants No. DMS-2005470, DMS-2305394 (first-named author) and DMS-2404599 (second-named author).

\section{The universal affine bundle}
	\label{sec:the_universal_affine_bundle}


\subsection{Construction and properties}
	\label{ssec:construction_and_properties}

Let $X$ be a compact $n$-dimensional complex manifold.
\subsubsection{Cohomology and homology groups}
	\label{sssec:setup_construction_and_properties}
For $0\leq p\leq n$, define the real Bott-Chern cohomology
$H^{p,p}(X;\bR)$ as the quotient of $d$-closed real $(p,p)$-forms modulo $dd^c$-exact ones, and the real Aeppli cohomology $H^{p,p}_{\rm A}(X;\bR)$ of $dd^c$-closed real $(p,p)$-forms modulo those of the form $\partial\alpha+\ov{\partial\alpha}$ where $\alpha$ is any $(p-1,q)$-form. It is well-known (see e.g. \cite[\S 2.1]{Ang}) that these spaces are finite-dimensional, and that there is a perfect pairing
$$H^{p,p}(X;\bR)\otimes H^{n-p,n-p}_{\rm A}(X;\bR)\to\mathbb{R},$$
given by cup product and integration, which thus identifies $H^{n-p,n-p}_{\rm A}(X;\bR)$ with $H^{p,p}(X;\bR)^\dual$.

\subsubsection{Currents}
	\label{sssec:Currents}

Let $\cA_{k,k}(X)$ denote the space of currents dual to smooth $(k,k)$-forms on $X$, or equivalently ``generalized $(n-k,n-k)$-forms''.
The operators $dd^c$ extends as a map $dd^c\colon \cA_{k+1,k+1}(X)\to \cA_{k,k}(X)$, and similarly for $d,\partial,\ov{\partial}$. We can then define the analogous spaces of Bott-Chern and Aeppli cohomologies $H_{p,p}(X;\bR)$ and $H^{\rm A}_{p,p}(X;\bR)$ using $(p,p)$-currents instead of smooth forms. As proved in \cite[Prop.2.2]{Big} (see also \cite[\S 2.1]{Ang}), these satisfy
$$H_{p,p}(X;\bR)\cong H^{n-p,n-p}(X;\bR),\quad H^{\rm A}_{p,p}(X;\bR)\cong H^{n-p,n-p}_{\rm A}(X;\bR),$$
where the isomorphism is induced by the natural inclusion of $(p,p)$-forms into $(n-p,n-p)$-currents.

We therefore set
\[
	V:=H_{1,1}^{\rm A}(X;\bR)\text{ and }V^\dual:=H^{1,1}(X;\bR).
\]
Of course, when $X$ is K\"ahler, both the Bott-Chern and Aeppli cohomology are isomorphic to the usual Dolbeault cohomology (of forms or currents).

Let $\cZ^{p,p}$ denote the sheaf of real, closed $(p,p)$-forms on $X$, so for any open subset $U\subset X$ we write $\cZ^{p,p}(U)$ for the smooth closed $(1,1)$-forms on $U$. Similarly, $\cZ^{p,p}_{dd^c}$ denotes the sheaf of real, $dd^c$-closed $(p,p)$-forms.

For $\alpha\in \cZ^{1,1}(X)$ denote by $[\alpha]$ its cohomology class, which is an element of $V^\dual$ by definition.

\begin{definition}[Affine bundle]
	\label{def:affine_bundle}
	Suppose that $E\to X$ is a smooth bundle with fibers diffeomorphic to a fixed vector space $V$.
	It will be called an \emph{affine bundle} over the vector space $V$ if it is equipped with a smooth map
	\[
		V\times E\to E
	\]
	which induces a simply-transitive action of $V$ on the fibers of $E$.
	The map will be denoted by $(v,e)\mapsto v+e$.
\end{definition}

\begin{remark}[Properties of affine bundles]
	\label{rmk:properties_of_affine_bundles}
	\leavevmode
	\begin{enumerate}
		\item Given two sections $s_1,s_2$ of an affine bundle $E$ defined on the open set $U\subset X$, their difference can be viewed as a function $(s_1-s_2)\colon U\to V$.
		\item Given a linear subspace $V'\subset V$ we can form the quotient space $E/V'$ which is an affine bundle over $V/V'$.
	\end{enumerate}
\end{remark}

\begin{definition}[Homogeneous functions]
	\label{def:homogeneous_functions}
	Suppose $E\to X$ is an affine bundle over the vector space $V$.
	Given a linear functional $[\alpha]\in V^\dual$, a function
	\[
		\phi \colon E\to \bR
	\]
	will be called \emph{$[\alpha]$-homogeneous} if
	\[
		\phi(x+v) = \phi(x)+\ip{[\alpha],v} \quad \forall v\in V
	\]
	where addition denotes the action of $V$ on $E$.
\end{definition}
Note that $0$-homogeneous functions are constant on fibers.

\subsubsection{Affine functions on affine bundles}
	\label{sssec:affine_functions_on_affine_bundles}
For later use, we record the following construction.
For an affine bundle $E\to X$, let $Lin(E)\to X$ denote the vector bundle of functions on the total space of $E$, which are affine on the fibers.
Let also $\bR_X$ denote the trivial vector bundle over $X$ with fiber $\bR$.
The rank of $Lin(E)$ is $\dim V+1$ and it sits in an exact sequence of vector bundles
\begin{align}
	\label{eqn:affie_functions_bdl_exact_seqn}
	0\to \bR_X \into Lin(E) \onto X\times V^\dual \to 0
\end{align}
where $\bR_X$ and $X \times V^\dual$ are trivial vector bundles.
The first inclusion is that of constant functions, which are clearly affine, and the quotient can be naturally identified with linear functions on $V$.

\subsubsection{Definition of the universal affine bundle}
	\label{sssec:definition_of_the_universal_affine_bundle}
When convenient, we will use the abbreviation
\[\cB_{1,1}(X)=\left\lbrace \beta\in\cA_{1,1}(X)\ |\ \beta=\partial\ov{S}+\overline{\partial} S, S\in\mathcal{A}_{1,2}(X)\right\rbrace,\]
for the space of $(1,1)$-currents of the form $\partial\ov{S}+\overline{\partial} S$, or equivalently of the currents which are $(1,1)$ components of $d$-exact real $2$-currents.
We will also denote by $P:\cA_{1,1}(X)\to\cA_{1,1}(X)/\cB_{1,1}(X)$ the projection onto the quotient (infinite-dimensional) vector space.

Fix a probability measure $\mu\in\cA_{0,0}(X)$. 
For each $x\in X$ we also have the Dirac-delta mass $\delta_x\in \cA_{0,0}(X)$.

Define now
\begin{align}
	\label{eqn:E_x_definition}
	\begin{split}
	E(x) & = P\left(\left\lbrace \beta \in \cA_{1,1}(X) \colon dd^c\beta = \delta_x - \mu \right\rbrace\right)\\ 
	& \text{and total space}\\
	E & = \bigcup_{x\in X} E(x) \xrightarrow{p}X,
	\end{split}
\end{align}
and observe that $E(x)$ has naturally the structure of an affine linear space, whose underlying vector space is finite-dimensional and isomorphic to $V$.

\begin{theorem}[Properties of universal bundle]
	\label{thm:properties_of_universal_bundle}
	Let $X$ be a compact complex manifold equipped with a probability measure $\mu$.
	The collection of fibers defined by \autoref{eqn:E_x_definition} has the structure of a smooth bundle, affine over the vector space $V=H_{1,1}(X)$ and with the following properties:
	\begin{description}
		\item[Pluriharmonic sections] There exists a subsheaf $\cP_E$ of the sheaf of smooth sections of $E$, such that if $s_1,s_2\in \cP_E(U)$ and $[\alpha]\in V^\dual$ then
		\[
			dd^c \ip{[\alpha],s_1-s_2} =0
		\]
		where $s_1-s_2$ is viewed as a $V$-valued function on $U$ and the pairing $\ip{[\alpha],s_1-s_2}$ is a smooth real-valued function on $U$.
		\item[Canonical potentials] There exists a continuous linear map of Fr\'echet spaces
		\[
			\Phi\colon \cZ^{1,1}(X)\to C^{\infty}(X;Lin(E))\subset C^\infty(E;\bR)
		\]
		such that for all $\alpha\in \cZ^{1,1}(X)$ the function $\Phi(\alpha)$ is $[\alpha]$-homogeneous.
		Furthermore, for any pluriharmonic local section $s\in \cP_E(U)$ we have that
		\begin{equation}\label{ddc}
			dd^c\left( s^*\Phi{(\alpha)} \right) = \alpha\vert_U.
		\end{equation}
        \item[Normalization] If $\alpha\in \cZ^{1,1}(X)$ satisfies $\alpha=dd^c\phi$ with $\int_X \phi\,\mathrm{d}\mu=0$ then $\Phi(\alpha)=p^{*}\phi$.
        \item[Universality for metrized line bundles] Suppose $L\to X$ is a holomorphic line bundle, with first Chern class $c_1(L)\neq 0$ in $H^{1,1}(X)$, and with metric $e^{\phi}$ viewed as $\phi\colon L^\times \to \bR$.
        Let also $\omega_\phi\in \cZ^{1,1}(X)$ denote its curvature form, normalized such that $c_1(L)=[\omega_\phi]$.
        Then there exists a unique ``classifying map''
        \[
        	cl_{\phi}\colon L^\times \to E/c_1(L)^{\perp}
        \]
        such that $cl^{*}_{\phi}\Phi\left(\omega_{\phi}\right)=\phi$.

		Furthermore, for another metric $\psi$ on $L$, the corresponding classifying map $cl_{\psi}$ is related to $cl_{\phi}$ by a global translation determined by the relation:
		\[
			\ip{c_1(L),(cl_{\psi} - cl_{\phi})} = \int_X (\psi - \phi) \, \mathrm{d}\mu
		\]
		where $\psi-\phi$ descends to a function on $X$.
	\end{description}
\end{theorem}

\begin{remark}[On properties of the universal bundle]
	\label{rmk:on_properties_of_the_universal_bundle}
	\leavevmode
	\begin{enumerate}
		\item Given $\alpha\in \cZ^{1,1}(X)$ the function $\Phi(\alpha)$ is $[\alpha]$-homogeneous, in particular it is invariant under $[\alpha]^\perp$, the subspace of homology classes that pair trivially against $[\alpha]$.
		In particular, $\Phi(\alpha)$ can be descended to the quotient $E/[\alpha]^\perp$.
		\item A consequence of the normalization property is that if $\alpha_1,\alpha_2\in \cZ^{1,1}(X)$ satisfy $\alpha_1=\alpha_2+dd^c \phi$ for some $\phi\in C^\infty(X,\mathbb{R})$, normalized by $\int_X \phi \,\mathrm{d}\mu=0$, then
		\[
			\Phi(\alpha_1) = \Phi(\alpha_2) + p^* \phi.
		\]
		\item When convenient (and not causing confusion) we will also write $\Phi(\alpha)=\phi_\alpha$.
		This function will be called ``the canonical potential of $\alpha$''.
	\end{enumerate}
\end{remark}

We now turn to the functoriality properties of $E$.
For a holomorphic map $X\xrightarrow{f}Y$ let $f^*\colon H^{1,1}(Y)\to H^{1,1}(X)$ and $f_*\colon H_{1,1}(X)\to H_{1,1}(Y)$ be the corresponding pullback and pushforward maps.
On the spaces of currents and differential forms the maps will be denoted in the same way.

\begin{theorem}[Functoriality of the affine bundle]
	\label{thm:functoriality_of_e}
	The bundle $E$ given by \autoref{eqn:E_x_definition} is functorial in the following sense.

	Suppose that $X\xrightarrow{f} Y $ is a holomorphic map between compact complex manifolds equipped with  probability measures $\mu_X,\mu_Y$, such that $f_*\mu_X = \mu_Y$.
	Then there is an induced smooth map $E_X \xrightarrow{f_E} E_Y$
	covering $f$, with the following properties.
	\begin{description}
		\item[Affine map] The map $f_E$ is a map of affine bundles, compatible with the map $f_*$ on $H_{1,1}$-spaces.
		Specifically
		\[
			f_E(e + v) = f_E(e) + f_*(v)\quad \forall e\in E_X, \forall v\in H_{1,1}(X).
		\]
		\item[Pullback of canonical potentials] For $\alpha\in \cZ^{1,1}(Y)$ we have that
		\[
			\Phi_X(f^*\alpha) = f_E^*\Big( \Phi_Y(\alpha) \Big).
		\]
	\end{description}
\end{theorem}

\begin{remark}[On naturality]
	\label{rmk:on_naturality}
	We will be interested in the group $\Aut(X,\mu)$ acting on $X$ via biholomorphic maps and preserving the probability measure $\mu$.
	The explicit naturality properties of our constructions under $\Aut(X,\mu)$ are:
	\begin{description}
		\item[Action on bundle] The group $\Aut(X,\mu)$ acts on $E$ on the left, by ``pushforward'', i.e.
		\[
			\gamma_1(\gamma_2(e)) = (\gamma_1\gamma_2)(e) \quad\forall \gamma_i\in \Aut(X,\mu), \forall e\in E
		\]
		\item[Action on potentials] For $\alpha\in \cZ^{1,1}(X)$ and $\phi_{\alpha}\colon E\to \bR$ constructed in \autoref{thm:properties_of_universal_bundle}, we have
		\[
			\phi_{\gamma^*\alpha} = \gamma^* \phi_\alpha
		\]
		where $\phi_{\gamma^*\alpha}$ is the canonical potential for $\gamma^*\alpha\in \cZ^{1,1}(X)$.
		\item[Pluriharmonic sections] The action of $\Aut(X,\mu)$ preserves the subsheaf of pluriharmonic sections. Indeed, this is a simple consequence of the identity
\[\langle [f^*\alpha],s_1-s_2\rangle = \langle[\alpha], f_*(s_1-s_2)\rangle.\]
	\end{description}
	Note in particular that the action of $\Aut(X,\mu)$ on $E$ is on the left, by pushforward, and on potentials and $(1,1)$-forms on the right, by pullback.

	Note that if $X,Y$ are complex manifolds equipped with $\mu_X$ and $\mu_Y$ which are smooth and strictly positive volume forms, then the condition $f_*\mu_X=\mu_Y$ implies that $f$ is surjective and $\dim Y\leq \dim X$.
\end{remark}



\subsection{Proofs of the properties}
	\label{ssec:proofs_of_the_properties}

We can now proceed to establish the results stated in \autoref{ssec:construction_and_properties}.

\subsubsection{Basic observations}
	\label{sssec:basic_observations_construction_bundle}
The fibers $E(x)$ defined in \autoref{thm:properties_of_universal_bundle} have a natural structure of affine space over $H_{1,1}(X;\bR)$, since the difference of any (equivalence classes) of currents $[\beta],[\beta']\in E(x)$ gives a unique element $[\beta-\beta']\in H_{1,1}(X;\bR)=:V$.

\subsubsection{Construction of canonical potential}
	\label{sssec:construction_of_canonical_potential}
Given any smooth $\alpha \in \cZ^{1,1}(X;\bR)$ the function $\Phi(\alpha)\colon E\to \bR$ is tautologically defined by
\[
	\Phi(\alpha)(\beta):=\ip{\beta,\alpha} \text{ for }\beta\text{ a representative in }E(x).
\]
Note that the value is well-defined independently of the representative $\beta$ since modifying it by $\beta\mapsto \beta'=\beta + \partial\ov{S}+\ov{\partial}S$ (with $S\in \cA_{1,2}(X)$) does not affect the value, since
\[
	\ip{\partial\ov{S}+\ov{\partial}S, \alpha} = \ip{(d(S+\ov{S})^{(1,1)},\alpha}= \ip{d(S+\ov{S}),\alpha}=\ip{S+\ov{S},d\alpha}=0,
\]
since $\alpha\in \cZ^{1,1}(X;\bR).$
\begin{definition}[Smooth sections of $E$]
	\label{def:smooth_sections_of_e}
	For an open set $U\subset X$, a section $s\colon U\to E\vert_U$ will be called \emph{smooth} if, for any smooth $\alpha\in\cZ^{1,1}(X)$ we have that $s^*\Phi(\alpha)$ is a smooth function on $U$.
\end{definition}
It follows from the constructions that the difference of two smooth sections is a smooth function $(s-s')\colon U\to V$, and that smooth sections form a sheaf of affine spaces.

\autoref{prop:properties_of_the_green_section} shows that there exist smooth sections, which are furthermore global sections of $E\to X$.
Given one such smooth global section $s\colon X\to E$, define a smooth structure on $E$ using the trivialization
\[
	X\times V\xrightarrow{s(x)+v}E
\]
which in particular shows the affine bundle is topologically trivial.

\subsubsection{Green section}
	\label{sssec:green_section}
By a classical result of Gauduchon \cite{Gau}, we can find a Hermitian metric $\omega$ on $X$ with $dd^c(\omega^{n-1})=0$ (for example, if $X$ is K\"ahler then every K\"ahler metric satisfies this), and by scaling we may assume that $\dVol:=\omega^n$ is a probability measure.

\begin{proposition}[Existence of Green functions]\label{greensect}
Given any point $x\in X$ there is a unique ``Green's function'' $u_x\in \mathcal{A}_{n,n}(X)$ (i.e. a distribution on $X$ satisfying $\int_X u_x\dVol=0$ and
\begin{equation}\label{green}
	dd^c\left(u_x \omega^{n-1}\right) = \delta_x - \mu.
\end{equation}
\end{proposition}
When $\mu$ is a smooth measure, then $u_x\in L^1(\dVol)$ and $u_x\omega^{n-1}$ is a Green current for the $0$-submanifold $x\in X$, as in the work of Gillet-Soul\'e \cite{GS}.
\begin{proof}
This result is rather standard, but we provide details for the reader's convenience. The second order differential operator
$$u\mapsto dd^c\left(u\omega^{n-1}\right)$$
is elliptic with kernel consisting precisely of constants (since $\omega$ is Gauduchon), see e.g. \cite{Gau}. Up to a numerical factor, its formal $L^2(\dVol)$ adjoint is the complex Laplacian
\[\Delta_\omega f=\frac{n \omega^{n-1}\wedge dd^c f}{\omega^n},\]
whose kernel are again constants, and whose image consists of functions $L^2$ orthogonal to constant.

A standard result in linear PDE (see e.g. the detailed writeup in \cite[Appendix A]{AS}) shows that we can find a Green's function $G_x\in L^1(\dVol)$ such that $dd^c(G_x\omega^{n-1})=\delta_x-\dVol$ as distributions. Adding a constant to $G_x$ does not change this equation because $\omega$ is Gauduchon, hence we may assume that  $\int_X G_x\dVol=0$, and $G_x$ is then unique. This is our desired $u_x$ in the case when $\mu=\dVol$, while for a general probability measure $\mu$ we argue as follows. By density of smooth forms among currents, we can find probability measures of the form $f_i\dVol$, $f_i\in C^\infty(X)$, which converge to $\mu$ weakly as distributions as $i\to\infty$. 
For each $i$, the function $f_i-1$ is thus $L^2(\dVol)$-perpendicular to the kernel of $\Delta_\omega$, hence by the Fredholm alternative and elliptic regularity we can find a smooth function $\widetilde{G}_i$ with average zero such that $dd^c(\widetilde{G}_i\omega^{n-1})=(f_i-1)\dVol$. Setting then $G_{i,x}:=G_x-\widetilde{G}_i\in L^1(\dVol)$, we see that
$dd^c(G_{i,x}\omega^{n-1})=\delta_x-f_i\dVol$.  We check that the sequence $G_{i,x}$ is weakly bounded in the space of distributions: given a smooth $(n,n)$-form, which we can write as $\chi\omega^n$ for a smooth function $\chi$, we need to check that there is $C>0$ such that
$$\left|\int_X G_{i,x}\chi\omega^n\right|\leq C,$$
for all $i$. For this, write $\hat{\chi}=\chi-\int_X\chi\omega^n$, and as said above we can find a smooth function $\psi$ with $\Delta_\omega\psi=\hat{\chi}$. Then we have
\[\begin{split}
\int_X G_{i,x}\chi\omega^n&=\int_X G_{i,x}\hat{\chi}\omega^n=\int_X G_{i,x}\Delta_\omega\psi\omega^n=n\int_X G_{i,x}dd^c\psi\wedge\omega^{n-1}\\
&=n\left(\psi(x)-\int_Xf_i\psi\omega^n\right),
\end{split}\]
and the RHS is bounded independent of $i$ since it converges to $n(\psi(x)-\int_X\psi \,\mathrm{d}\mu)$.
The sequence $G_{i,x}$ is thus weakly bounded, and hence by standard theory of distributions (see e.g. \cite[P.178, Ex.9]{rudin}) it has a subsequence that weakly converges to a distribution $u_x$, which satisfies $dd^c(u_x\omega^{n-1})=\delta_x-\mu$ and $\int_X u_x\dVol=0$, as desired. Uniqueness of $u_x$ follows as above.
\end{proof}


In particular this shows that the fiber $E(x)$ in \autoref{eqn:E_x_definition} is not empty.
Put together, the Green functions give a section
\[
	s_{\omega}\colon X \to E\quad s_\omega(x):=u_x\omega^{n-1}.
\]

\begin{proposition}[Properties of the Green section]
	\label{prop:properties_of_the_green_section}
	Let $s_\omega\colon X\to E$ be the Green section.
	\begin{enumerate}
		\item It is a smooth section according to \autoref{def:smooth_sections_of_e}.
		\item If $\alpha=dd^c\phi$ with $\int_X \phi \,\mathrm{d}\mu =0$ then
		\[
			s_{\omega}^*\Phi(\alpha)(x)=\phi(x).
		\]
	\end{enumerate}
\end{proposition}
\begin{proof}
	Let $\alpha\in \cZ^{1,1}(X)$ be a smooth closed $(1,1)$-form.
	We must show $\rho(x):=\int_X u_x \omega^{n-1}\wedge\alpha$ is a smooth function on $X$.
	The scalar $c:=\int_X \omega^{n-1}\wedge \alpha$ has the property that the smooth function
	\[
		f=\left[\frac{\omega^{n-1}\wedge\alpha}{\omega^{n}}-c\right] \text{ satisfies }\int_X f \dVol = 0.
	\]
We can then find a smooth function $\psi$ with $\int_X \psi\,\mathrm{d}\mu=0$ satisfying $\frac{1}{n}\Delta_\omega\psi=f$.
	For any $x\in X$, using again that $\int_X u_x \dVol = 0$ and $dd^c(u_x\omega^{n-1})=\delta_x-\mu$, we find that
	\[\begin{split}
		\rho(x) &= \int_X u_x \left(\frac{\omega^{n-1}\wedge\alpha}{\omega^n}-c\right)\dVol = \int_X u_x f \dVol\\
&=\frac{1}{n}\int_X u_x\Delta_\omega\psi \dVol=\int_X u_x dd^c\psi\wedge\omega^{n-1}=\psi(x),
	\end{split}\]
which shows that $\rho$ is smooth.

	Part (ii) follows from the calculation
	\[
		\Phi(\alpha)(s_{\omega}(x))
		=
		\int_{X}\left(dd^c \phi\right)u_x \omega^{n-1}
		=
		\int_X \phi\cdot  dd^c\left(u_x \omega^{n-1}\right)
		= \phi(x).
	\]
\end{proof}

\subsubsection{Defining the pluriharmonic sections}
	\label{sssec:defining_the_pluriharmonic_sections}
In order to specify the subsheaf $\cP_E$ of smooth sections of $E$ which we call ``pluriharmonic sections'', it suffices to define $\cP_E(U)$ for each sufficiently small open set $U$ (since these form a basis of the topology on $X$). Furthermore, it suffices to specify a single such section that has the properties required in \autoref{thm:properties_of_universal_bundle}.
Indeed the difference of any two sections $s_1,s_2$ of $E$ is a vector-valued function $s_1-s_2\colon U\to V$, and for these pluriharmonicity is defined in the usual way (namely, for every $[\alpha]\in V^\dual,$ the function $\langle[\alpha],s_1-s_2\rangle$ is pluriharmonic), and property \eqref{ddc} is preserved since
\[
dd^c((s_1-s_2)^*\Phi(\alpha))=dd^c\langle s_1-s_2,[\alpha]\rangle=0.
\]

Fix $\alpha_1,\cdots, \alpha_h\in \cZ^{1,1}(X)$ such that their cohomology classes give a basis of $H^{1,1}(X)$, and fix $\alpha_1^\dual,\cdots, \alpha_h^{\dual}\in \cZ^{n-1,n-1}_{dd^c}(X)$ such that their cohomology classes give a dual basis in $H^{n-1,n-1}_{\rm A}(X)\cong H_{1,1}(X)$.
Then if the open set $U$ is sufficiently small there exist smooth functions $\psi_i\in \cA^0(U)$ such that $dd^c \psi_i = \alpha_i$.
Consider now the smooth functions
\[
	f_i(x):= \psi_i(x) - \int_{X}u_x \omega^{n-1}\wedge \alpha_i
	\quad
	\text{ defined on }U,
\]
where smoothness follows from the argument in \autoref{prop:properties_of_the_green_section}.
Let us check that the section $s_U\colon U\to E\vert_U$ given by
\[
	s_{U}(x):= s_\omega(x) + \sum {f_i}(x)\cdot \alpha_i^\dual \in E(x)
\]
satisfies the requirements of a pluriharmonic section.

Given any $\alpha\in \cZ^{1,1}(X)$ there exist unique scalars $c_i\in \bR$ and smooth function $\phi\in \cA^0(X)$ with $\int_X \phi \,\mathrm{d}\mu=0$ such that $\alpha = dd^c \phi + \sum c_i \alpha_i$.
Then the pullback of its canonical potential $\Phi(\alpha)$ under the section $s_U$ satisfies:
\begin{align*}
	s_{U}^{*}\Phi(\alpha)(x) & =
	\int_{X}\left(u_x \omega^{n-1}+ \sum_{i} f_i(x)\alpha_i^\dual\right)\wedge
	\left(dd^c \phi + \sum_{j} c_j \alpha_j\right)\\
	& = \phi(x) + \sum_{j} c_j\left(\int_{X}u_x \omega^{n-1}\wedge \alpha_j\right)
	+ \sum_{i,j} c_j f_i(x) \int_X \alpha_i^\dual\wedge \alpha_j\\
	& = \phi(x) + \sum_{j} c_j\left(\int_{X}u_x \omega^{n-1}\wedge \alpha_j\right)
	+ \sum_{j} c_j f_j(x)\\
	& = \phi(x) + \sum c_j \psi_j(x)
\end{align*}
where to deduce the last line we used that $\left(\int_X u_x \omega^{n-1}\wedge\alpha_j\right) + f_j(x) = \psi_j(x)$.
This calculation shows that $s_U$ satisfies the requirements of a pluriharmonic section since $dd^c\left[s_{U}^{*}\Phi(\alpha)\right]=\alpha\vert_U$.

Note that the pluriharmonic sections are a subsheaf of the smooth sections, since locally they differ from the Green section by smooth $V$-valued functions.

\subsubsection{Checking functoriality properties}
	\label{sssec:checking_functoriality_properties}
Assume the setup of \autoref{thm:functoriality_of_e}.
Given $x\in X$ with $y=f(x)\in Y$, if $\beta\in \cA_{1,1}(X)$ is such that $dd^c \beta=\delta_x - \mu_X$, then
\[
	dd^c (f_*\beta)= f_*\left(dd^c \beta\right) = f_*(\delta_x - \mu_X)=\delta_y - \mu_Y
\]
and $f_*(\cB_{1,1}(X))\subset \cB_{1,1}(Y)$ (since $f$ is holomorphic so $f_*$ commutes with $\partial,\ov{\partial}$), so the fiberwise map
\[
	f_{E,x}\colon E(x)\to E(f(x)) \text{ is well-defined.}
\]
From the construction it is also immediate that the map is affine relative to the map of vector spaces $f_*\colon H_{1,1}(X)\to H_{1,1}(Y)$.

To check compatibility with canonical potentials, we compute:
\begin{align*}
	\Phi_X\left(f^*\alpha\right)([\beta]) & =
	\ip{\beta, f^*\alpha} = \ip{f_*\beta,\alpha}\\
	& = \Phi_Y(\alpha)(f_{E}([\beta]))
\end{align*}
as required.

\subsubsection{Smoothness of pushforward map}
	\label{sssec:smoothness_of_pushforward_map}
To check that $f_E$ is smooth, it suffices to check that the pullback of local coordinate functions on $E_Y$ are smooth functions on $E_X$.
If we fix a collection $\alpha_1,\ldots,\alpha_h\in \cZ^{1,1}(Y)$ whose cohomology classes form a basis of $H^{1,1}(Y)$, and additionally functions $\phi_1,\cdots, \phi_n \in \cA^{0}(Y)$ (with $n=\dim Y$) such that the $\phi_i$ are coordinates near $y=f(x)$, then the canonical potentials $\Phi_Y(\alpha_i)$ and $\Phi_Y(dd^c\phi_j)$ give coordinate functions on $E_Y$ near $y$.
By compatibility of pullback with the canonical potentials established in \autoref{sssec:checking_functoriality_properties}, it follows that the pullbacks of the coordinate functions are smooth on $E_X$, so the map $f_E$ is smooth.

\subsubsection{Construction and properties of the classifyig map}
	\label{sssec:Construction-and-properties-of-the-classifyig-map}
Suppose now that $L\to X$ is a holomorphic line bundle over $X$, with metric $\phi\colon L^\times \to \bR$, and with $c_1(L)\neq 0$ in $H^{1,1}(X)$.
Since $\phi$ satisfies $\phi(\lambda\cdot z)=\log|\lambda| + \phi(z)$ on the fibers, and $\Phi(\omega_\phi)$ defines a nontrivial affine function on the fibers of $E/[\omega_\phi]^\perp$, there is a unique map
\[
	cl_\phi \colon L^\times \to E/[\omega_\phi]^{\perp}
\]
such that $cl_\phi^*(\Phi(\omega_\phi)) = \phi$.
Moreover, $cl_\phi$ satisfies:
\[
	\ip{c_1(L),cl_\phi(\lambda \cdot z) - cl_\phi(z)} = \log |\lambda| \quad
	\forall \lambda\in \bC^\times, \forall z\in L^\times.
\]
Now suppose that $\psi=\phi + p^* f$ is another metric on $L$, where $f$ is a smooth function on $X$.
Then we know that
\[
	\Phi(\omega_{\psi}) = \Phi(\omega_\phi) + p^* \hat{f} \quad \text{on }E/c_1(L)^\perp
\]
where $\hat{f}=f - \int_X f d\mu$.
From the defining properties of $cl_\phi$ and $cl_{\psi}$, subtracting the corresponding equations, we find that
\[
	\ip{c_1(L),cl_{\psi} - cl_\phi} = p^* \left(f - \hat{f}\right) = \int_X fd\mu
\]
and this concludes the proof.

\begin{remark}[A heuristic for the construction of the universal bundle]
	\label{rmk:a_heuristic_for_the_construction_of_the_universal_bundle}
	The definition of the bundle $E$ in \autoref{eqn:E_x_definition} is inspired by the canonical isometric embedding (Kuratowski embedding) of a metric space $X$ into the space of Lipschitz functions on $X$ given by $x\mapsto \dist(-,x)$.
	This embedding can be used to compactify $X$, by taking the quotient modulo constants or by introducing a reference point $x_0\in X$ and using $x\mapsto \dist(-,x)-\dist(x_0,x)$ instead. See e.g. \cite[$3\frac{1}{2}.36$]{Gr}

	The starting point is the short exact sequence of sheaves
	\[
	0\to\mathcal{P}\to \mathcal{A}^{0}\xrightarrow{dd^c} \cZ^{1,1}\to 0
	\]
	where $\mathcal{P}$ denotes the sheaf of real pluriharmonic functions, $\mathcal{A}^0$ the sheaf of smooth functions, $\cZ^{1,1}$ the sheaf of $d$-closed real $(1,1)$ forms.
	The long exact sequence in cohomology becomes:
	\begin{equation}
		\label{eqn:les_pluriharmonic_sheaves}
	0\to \mathbb{R}\to C^\infty(X,\mathbb{R})\to \cZ^{1,1}(X)\to H^1(X,\mathcal{P})\to 0,
	\end{equation}
	and there is a natural isomorphism $H^1(X,\mathcal{P})\cong H^{1,1}(X;\mathbb{R})$ (see e.g. \cite[Corollary 2]{HarveyLawson_Intrinsic}).

	Informally dualizing \autoref{eqn:les_pluriharmonic_sheaves} we find
	\begin{align}
		\label{eqn:basic_LES_currents}
		0\from \mathbb{R}\from
		\cA_{0,0}(X;\mathbb{R})\xleftarrow{dd^c} \cZ_{1,1}(X)\from H_{1,1}(X;\bR)\from 0
	\end{align}
	where $\cZ_{1,1}(X)$ stands for $\cA_{1,1}(X)/\cB_{1,1}(X)$ and
	\[
		\cB_{1,1}(X)=\{\eta\in\mathcal{A}_{1,1}(X)\ |\ \langle\eta,\alpha\rangle=0, \textrm{ for all }\alpha\in \cZ^{1,1}(X)\},
	\]
    and the map $dd^c:\cZ_{1,1}(X)\to \cA_{0,0}(X;\bR)$ is the $dd^c$ operator on currents.
	But the space $\cB_{1,1}(X)$ is in fact equal to the space of $(1,1)$ components of $d$-exact currents, or even more explicitly
	\[
	\cB_{1,1}(X)=\{\eta\in\mathcal{A}_{1,1}(X;\bR)\ |\ \eta=\del \ov{S} + \delbar S, \textrm{ for some }S\in\mathcal{A}_{1,2}(X;\bR)\},		
	\]
	see e.g. \cite[p.267]{Lamari_Courants}, and the quotient of $dd^c$-closed real $(1,1)$ currents modulo $\cB_{1,1}(X)$ is the Aeppli cohomology group $H_{1,1}(X;\bR)$ which as we explained is naturally the dual of $H^{1,1}(X;\mathbb{R})$.

	The embedding $X\into \cA_{0,0}(X;\bR)$ given by $x\mapsto \delta_x - \mu$ allows us to ``pullback'' the universal $H_{1,1}$-bundle from the exact sequence in \autoref{eqn:basic_LES_currents}.
\end{remark}

					\section{Examples and Applications}
					\label{sec:Examples-and-Applications}


\subsection{Examples of varieties with invariant measures}
	\label{ssec:examples_of_varieties_with_invariant_measures}

We now collect some examples of complex manifolds and automorphism groups that preserve a measure.

If the group of holomorphic automorphisms $\Aut(X)$ is amenable, then it always preserves a probability measure when acting on a compact space, see \cite[\S4]{Zimmer1984_Ergodic-theory-and-semisimple-groups} for many equivalent characterizations of amenability.
We now consider conditions on $X$ that might yield a natural invariant measure.

\begin{corollary}[Calabi--Yau case]
	\label{cor:calabi_yau_case}
	Let $X$ be Calabi--Yau, in the sense that $X$ is a compact complex $n$-fold with  torsion canonical bundle.
	Let $\Omega$ be a trivializing section of $\ell K_X, \ell\geq 1$, normalized such that $\mu:=\sqrt{-1}^{n^2} (\Omega \wedge \ov{\Omega})^{\frac{1}{\ell}}$ has integral $1$.

	Then the action of the group of biholomorphic automorphisms $\Aut(X)$ lifts to the bundle $E$ and the projection map $p\colon E\to X$ is equivariant.
	Furthermore, all the constructions are natural for the action of $\Aut(X)$ (see \autoref{rmk:on_naturality}).
\end{corollary}
\begin{proof}
	We have $f_*\mu=\mu$ for all $f\in \Aut(X)$, and so the functoriality properties in \autoref{thm:functoriality_of_e} yield the result.
\end{proof}

Generalizing \autoref{cor:calabi_yau_case}, we have the following known result which appears in \cite[Theorem 3.1]{sakai}:

\begin{proposition}\label{nonnegkod}
Let $X$ be a compact complex manifold $X$ with nonnegative Kodaira dimension, which by definition means that $H^0(X,mK_X)\neq 0$ for some $m\geq 1$. Then $X$ admits a $\mathrm{Bim}(X)$-invariant probability measure $\mu$. This measure is absolutely continuous with respect to Lebesgue, with continuous density.
\end{proposition}
This is a generalization of the Calabi-Yau volume form in \autoref{cor:calabi_yau_case}, and was first discovered by Narasimhan-Simha \cite{NS} when $K_X$ is ample. 
For the reader's convenience we give the short proof.
\begin{proof}
Fix a value of $m\geq 1$ for which $H^0(X,mK_X)\neq 0$ and consider the $L^{\frac{2}{m}}$ pseudonorm on $H^0(X,mK_X)$ defined by
$$\|s\|_{m}=i^{n^2}\left(\int_X(s\wedge\ov{s})^{\frac{1}{m}}\right)^m,$$
which is not induced by a Hermitian inner product when $m\geq 2$.
Nonetheless we can use it to define the $m$th Bergman kernel $(n,n)$-form
$$B_m(x)=\sup_{\|s\|_m=1}i^{n^2}(s\wedge\ov{s})^{\frac{1}{m}}(x),$$
which is continuous on $X$ (see \cite{NS} or \cite[Lemma 2.2 (2)]{taka}), nonnegative and strictly positive on the nonempty Zariski open set whose complement is the base locus of $mK_X$ (the intersection of all the zero loci of all sections of $mK_X$).
We thus define the probability measure
$$\mu=\frac{B_m}{\int_X B_m}.$$
Given $f\in\mathrm{Bim}(X)$ and $s\in H^0(X,mK_X)$, we can define its pullback $f^*s\in H^0(X,mK_X)$ by first pulling back via $f$ away from the indeterminacy locus of $f$ (which is a closed analytic subset of codimension at least $2$ in $X$) and then extending the pullback section across the indeterminacy using Hartogs' Theorem, and the resulting map $f^*:H^0(X,mK_X)\to H^0(X,mK_X)$ is an isomorphism.
Thus, we can also pull back $B_m$ (and $\mu$) via $f$, and noticing that clearly
$$\|f^*s\|_{m}=\|s\|_m,$$
we see that
$$f^*B_m=\sup_{\|s\|_m=1}i^{n^2}(f^*s\wedge\ov{f^*s})^{\frac{1}{m}}=\sup_{\|f^*s\|_m=1}i^{n^2}(f^*s\wedge\ov{f^*s})^{\frac{1}{m}}=B_m,$$
which proves the $\mathrm{Bim}(X)$-invariance of $\mu$.
\end{proof}

The assumption of nonnegativity of the Kodaira dimension cannot be dropped in general, as shown in \S \ref{sssec:failure_for_negative_kodaira_dimension}. However, there is a weaker condition which is sufficient to obtain a probability measure invariant under the group of pseudoautomorphisms, as we now explain. We say that the canonical bundle $K_X$ of a compact complex manifold $X$ is pseudoeffective if it admits a singular Hermitian metric with nonnegative curvature form. Equivalently, fixing a smooth Hermitian metric $h$ on $K_X$ with curvature form $\alpha$, which is a closed real $(1,1)$-form cohomologous to $c_1(K_X)$, pseudoeffectivity means that there exists a quasi-psh function $\vp$ on $X$ with $\alpha+dd^c\vp\geq 0$ in the weak sense.

If $X$ has nonnegative Kodaira dimension then $K_X$ is pseudoeffective: it suffices to take any $m\geq 1$ for which $H^0(X,mK_X)\neq 0$, take a nontrivial section $s\in H^0(X,mK_X)$ and let $\vp=\frac{1}{m}\log|s|^2_h$, which is quasi-psh and satisfies $\alpha+dd^c\vp\geq 0$. The converse implication, that $K_X$ pseudoeffective should imply that $X$ has nonnegative Kodaira dimension (when $X$ is projective), is a well-known open problem in birational geometry, known as the Nonvanishing Conjecture, see e.g. \cite{DHP}. Furthermore, for compact K\"ahler manifolds, $K_X$ being pseudoeffective is equivalent to $X$ not being uniruled: this remarkable geometric characterization was proved in \cite{BDPP} for projective manifolds, and in \cite{Ou} for compact K\"ahler manifolds. On the other hand, there are counterexamples to the Nonvanishing Conjecture for general compact complex manifolds (see e.g. \cite[Ex.3.1]{nkcy}), so in general $K_X$ pseudoeffective is strictly weaker than $X$ having nonnegative Kodaira dimension.

\begin{proposition}\label{kxpsef}
Let $X$ be a compact complex manifold $X$ whose canonical bundle $K_X$ is pseudoeffective. Then $X$ admits a $\mathrm{PsAut}(X)$-invariant probability measure $\mu$.
This measure is absolutely continuous with respect to Lebesgue, with bounded density.
\end{proposition}
When $X$ is a projective manifold, the two measures constructed in \autoref{nonnegkod} and \autoref{kxpsef} are directly related, see \cite[Proposition 5.19]{BD}.
\begin{proof}
The measure $\mu$ is Tsuji's ``supercanonical measure'' constructed in \cite{Tsuji}. We follow here the exposition of the construction in \cite[\S 5]{BD}, although the $\mathrm{PsAut}(X)$-invariance of $\mu$ does not appear in these references. Fix a Gauduchon metric $\omega$ on $X$, and let $h$ be the metric it induces on $K_X$, whose curvature equals $\alpha:=-\mathrm{Ric}(\omega)$. If $\vp$ is a quasi-psh function with $\alpha+dd^c\vp\geq 0$, then $he^{-\vp}$ is a singular Hermitian metric on $K_X$ with nonnegative curvature, and we define
$$\vp_{\rm can}(x)=\sup\{\vp(x)\ |\ \alpha+dd^c\vp\geq 0, \int_X e^\vp\omega^n\leq 1\}.$$
Since $K_X$ is pseudoeffective, $\vp_{\rm can}\not\equiv-\infty$. Next, we claim that $\vp_{\rm can}$ is bounded above. For this, we need the existence of a Green's function $G_x$ (for any $x\in X$), which is in $L^1(\omega^n)$ with average zero and bounded below by some constant $-A$ (see again \cite[Appendix A]{AS}), such that if $\vp$ is any function as in the definition of $\vp_{\rm can}$, then we have
\[\vp(x)=\int_X\vp\omega^n -\int_XG_x\Delta_\omega \vp  \omega^n \leq \int_Xe^{\vp}\omega^n -\int_X(G_x+A) \Delta_\omega \vp \omega^n,\]
using that $\int_X \Delta_\omega\vp\omega^n=0$ since $\omega$ is Gauduchon. But then we have $G_x+A\geq 0$ and $\tr_{\omega}{\alpha}+\Delta_\omega\vp\geq 0$, so
\[\vp(x)\leq 1+\int_X\tr_{\omega}{\alpha}(G_x+A) \omega^n\leq 1+\|\tr_{\omega}{\alpha}\|_{L^\infty(X)} A \int_X\omega^n,\]
and the RHS is uniformly bounded independent of $\vp$ and $x$. Thus all the functions $\vp$ in the supremum in the definition of $\vp_{\rm can}$ have a uniform upper bound, and so $\vp_{\rm can}$ is bounded above. It is then easy to conclude (see \cite[\S 5]{BD}) that $\vp_{\rm can}$ is a quasi-psh function with $\alpha+dd^c\vp_{\rm can}\geq 0$. Since $\vp_{\rm can}$ is bounded above and not identically $-\infty$, we have $0<\int_Xe^{\vp_{\rm can}}\omega^n<\infty$ (although it has no reason to be equal to $1$). We then define
$$\mu=\frac{e^{\vp_{\rm can}}\omega^n}{\int_Xe^{\vp_{\rm can}}\omega^n}.$$
We need to show that for any $T\in\mathrm{PsAut}(X)$ we have $T_*\mu=\mu$. Since $T$ is a pseudoautomorphism, we can find closed analytic subvarieties $I,J\subset X$ of codimension at least $2$ such that $T:X\backslash I\to X\backslash J$ is well-defined and a biholomorphism.
Then on $X\backslash I$ we have that $T^*\omega^n$ is a smooth positive volume form, and we can write
\begin{equation}\label{pullb}
T^*\omega^n=e^\psi\omega^n,
\end{equation}
for some smooth function $\psi$ on $X\backslash I$, with $\int_{X\backslash I}e^\psi\omega^n=\int_X\omega^n$. Taking $dd^c\log$ of \autoref{pullb} we see that on $X\backslash I$ we have
$$T^*\alpha = \alpha+dd^c\psi.$$
If $\vp$ is any quasi-psh function with $\alpha+dd^c\vp\geq 0$ and $\int_X e^\vp\omega^n\leq 1$, then on $X\backslash I$ we have
$$\alpha+dd^c(T^*\vp+\psi)=T^*(\alpha+dd^c\vp)\geq 0,$$
and so $T^*\vp+\psi$ is quasi-psh on $X\backslash I$. By the Grauert-Remmert extension theorem \cite{grauert_remmert}, $T^*\vp+\psi$ extends to a quasi-psh function on $X$ (denoted in the same way), which still satisfies $\alpha+dd^c(T^*\vp+\psi)\geq 0$ weakly on $X$. We also clearly have
$$\int_X e^{T^*\vp+\psi}\omega^n=\int_X e^{T^*\vp}T^*\omega^n=\int_Xe^{\vp}\omega^n\leq 1,$$
and so on $X\backslash I$ we have
\[\begin{split}
T^*\vp_{\rm can}+\psi&=T^*\sup\{\vp\ |\ \alpha+dd^c\vp\geq 0, \int_X e^\vp\omega^n\leq 1\}+\psi\\
&=\sup\{T^*\vp+\psi\ |\ \alpha+dd^c\vp\geq 0, \int_X e^\vp\omega^n\leq 1\}\\
&\leq \sup\{T^*\vp+\psi\ |\ \alpha+dd^c(T^*\vp+\psi)\geq 0, \int_X e^{T^*\vp+\psi}\omega^n\leq 1\}\\
&=\sup\{\eta\ |\ \alpha+dd^c\eta\geq 0, \int_X e^{\eta}\omega^n\leq 1\}=\vp_{\rm can},
\end{split}\]
and since $I$ has Lebesgue measure zero, this implies that
$$T^*\mu\leq \mu,$$
and applying the same argument to $T^{-1}$ concludes the proof (recall that $T_*=(T^{-1})^*$).
\end{proof}

\subsubsection{Failure for negative Kodaira dimension}
	\label{sssec:failure_for_negative_kodaira_dimension}
\autoref{nonnegkod} fails in general when $X$ has negative Kodaira dimension. For example let $X=\mathbb{P}^n(\bC)$, and suppose that we had an $\Aut(\mathbb{P}^n)$-invariant probability measure $\mu$.
On the one hand, we can fix a positive-definite hermitian product on $\bC^{n+1}$ and consider the associated unitary group $H:=\bbU(\bC^{n+1})$.
Any $\Aut(\bP^n)$-invariant probability measure on $\bP^n(\bC)$ will then have to be the homogeneous one induced by the Fubini--Study metric, since $\bP^n(\bC)\isom H/H'$ where $H'=\bbU(\bC)\times\bbU(\bC^n)$.
On the other hand, this measure is not preserved by $\PGL_{n+1}(\bC)$.

\subsubsection{Remark}
Robert Berman pointed out to us that in \cite{Ber1} he gives another probabilistic construction of an $\mathrm{Aut}(X)$-invariant probability measure on any compact K\"ahler manifold $X$ with $K_X$ effective. 

	\subsection{Universal torsor}
		\label{ssec:Universal-torsor}

We now collect an application to the action of $\Aut(X)$ on holomorphic line bundles.

\subsubsection{Setup}
	\label{sssec:Setup-universal_torsor}

Suppose $X$ is an algebraic variety over $\bC$.
Assume that $\Pic^0(X)=0$ and set $N_{\bZ}:=\Pic(X)$, then taking the first Chern class yields a map $N_{\bZ}\to H^{1,1}$ with kernel the finite torsion subgroup of $N_{\bZ}$.
For simplicity, we will assume that the torsion subgroup is trivial.

\begin{definition}[Universal torsor]
	\label{def:Universal-torsor}
	The \emph{\Neron--Severi torus} of $X$ is the algebraic torus $\bbT_X$ with character group $N_{\bZ}$, i.e. $\Hom(\bbT_X,\bbG_M)=N_{\bZ}$.
	
	A \emph{universal torsor} is a principal $\bbT_X$-bundle $\cU\to X$, with the property that for any $\chi\in N_{\bZ}$ and line bundle $L_{\chi}$, the space $\cU{\times}_{\bbT_X} L_{\chi}$ is isomorphic to $L_{\chi}^\times$.
\end{definition}
\noindent We recall that for a line bundle $L$ on $X$, the space $L^\times$ is the complement of the zero section in $L$.

\begin{remark}[On universal torsors]
	\label{rmk:On-universal-torsors}
	\leavevmode
	\begin{enumerate}
		\item
		Any two universal torsors are isomorphic, with isomorphism unique up to translation by an element of $\bbT_X$.
		\item
		The discussion works over any base field $k$, using the Picard group over an algebraic closure $\ov{k}$, which is equipped with a Galois action, to obtain a torus $\bbT_{X,\ov{k}}$ with a Galois action, which can therefore be descended to a torus over $k$.
		\item
		To deal with torsion in the Picard group, one can either pass to a finite-index subgroup of it, or allow $\bbT_X$ to have finitely many connected components.
	\end{enumerate}
	See Colliot-Th\'{e}l\`ene and Sansuc \cite{Colliot-TheleneSansuc1976_Torseurs-sous-des-groupes-de-type-multiplicatif;-applications-a-letude} and Hassett and Tschinkel \cite{HassettTschinkel2004_Universal-torsors-and-Cox-rings} for background on universal torsors.
\end{remark}

\subsubsection{A construction and action}
	\label{sssec:a_torsor}
Let us see first that a universal torsor always exists. In the case of a single automorphism, this construction appears in \cite[\S 3.1]{cantat}.

Fix a set $L_1,\ldots, L_{\rho}$ of line bundles whose first Chern classes provide a basis of $N_{\bZ}$, which gives us an isomorphism $\bZ^\rho \to N_{\bZ}$.
Set now:
\[
	\cU:= L_1^{\times} \times_X \dots \times_X L_{\rho}^{\times}
\]
where the product of line bundles is fibered over $X$.
This is a torsor over $(\bC^{\times})^\rho$, which is then viewed as a principal $\bbT_X$-bundle via the isomorphism $\bZ^\rho \to N_{\bZ}$.
It is clearly a universal torsor.

Suppose now $\gamma\in \Aut(X)$ is an automorphism of $X$.
The universal torsors $\gamma_*\cU$ and $\cU$ are isomorphic, since $\gamma_* L_i$ still yield a basis of $N_{\bZ}$.
Indeed, $\gamma_* L_i^\times \isom \otimes (L_j^\times)^{a_{i,j}}$ where $(a_{i,j})$ is the matrix of $\gamma$ acting on $N_{\bZ}$ in the basis provided by the $L_i$.
We can thus choose a map:
\[
	c_{\gamma}\colon \gamma_* \cU \toisom \cU \quad \text{for any }\gamma\in \Aut(X)
\]
and we have a commutative diagram:
\begin{equation}
	\label{eqn_cd:action_and_isomorphism}
\begin{tikzcd}
	\cU
	\arrow[r, "\gamma_*", swap]
	\arrow[d, ""]
	\arrow[rr, bend left, "\tilde{\gamma}"]
	&
	\gamma_* \cU
	\arrow[r, "c_\gamma", swap]
	\arrow[d, ""]
	& \cU
	\arrow[d]
	\\
	X
	\arrow[r, "\gamma"]
	&
	X
	\arrow[r, "\id"]
	& X
\end{tikzcd}
\end{equation}
We set $\tilde{\gamma}$ to be the composition of the natural map $\cU\to \gamma_* \cU$ and $c_{\gamma}$.

Note that for two automorphisms $\gamma_1,\gamma_2$ with product $\gamma_3:=\gamma_1\gamma_2$ we have that $\tilde{\gamma}_3 \cdot \tilde{\gamma}_2^{-1}\tilde{\gamma}_1^{-1}$ is an automorphism of $\cU$ that acts trivially on the base $X$, therefore it acts by an element of $\bbT_X$.

\begin{theorem}[Reduction to unitary part]
	\label{thm:reduction_to_unitary_part}
	Let $\bbT_X^c\subset \bbT_X$ denote the product of unit circles, i.e. the compact subtorus.
	Suppose that $\Gamma$ is a subgroup of $\Aut(X)$, preserving a probability measure $\mu$ on $X$.
	
	Then, it is possible to choose the isomorphisms $c_{\gamma}$ in \autoref{eqn_cd:action_and_isomorphism} such that $\forall \gamma_1,\gamma_2\in \Gamma$ with $\gamma_3:=\gamma_1\gamma_2$ we have that
	\[
		\tilde{\gamma}_3 \cdot \tilde{\gamma}_2^{-1}\tilde{\gamma}_1^{-1} \in \bbT_X^c.
	\]
\end{theorem}

\begin{proof}
	We will make use of the affine bundle $E\to X$ constructed in \autoref{thm:properties_of_universal_bundle}, using the measure $\mu$.

	Consider the maps dual to the inclusion $N_\bR\into H^{1,1}$, and the associated kernel:
	\[
		M_{\bR}:=N_{\bR}^{\dual}\twoheadleftarrow H_{1,1}\hookleftarrow K
	\]
	Define next the quotient
	\[
		E_{M}:=E/K \text{ which is an affine bundle over }M_{\bR}.
	\]
	In the construction of the universal torsor $\cU$ in \autoref{sssec:a_torsor}, equip each $L_i$ with a smooth metric $\phi_i$.
	Denote by $\phi$ the function on $\cU$ obtained from the metrics $\phi_i$ on $L_i$.
	By the universality of the affine bundle $E$ for metrized holomorphic line bundles provided by \autoref{thm:properties_of_universal_bundle}, we have a classifying map
	\[
		cl_{\phi}\colon \cU \to E_M=E/K
	\]
	whose fibers are compact $\rho$-dimensional tori.
	By applying $\gamma_*$ we obtain maps
	\[
		cl_{\gamma_*\phi}\colon \gamma_* \cU \to \gamma_* E_M
	\]
	The arbitrary initial isomorphism $c_{\gamma}$ between $\gamma_* \cU$ and $\cU$ need not commute with the natural map $\gamma_*E_N\to E_N$ provided by \autoref{thm:functoriality_of_e}.
	However, because the classifying map $cl_{\phi}$ is unique, and for another metric such as $\phi':=(c_\gamma)_*(\gamma_*\phi)$, the maps $cl_\phi$ and $cl_{\phi'}$ differ by a translation by an element of $M_{\bR}$, we see that the diagram \autoref{eqn_cd:archimedean_torsor_compatibility} commutes up to a translation by an element of $M_{\bR}$.
	We can then adjust $c_\gamma$ by an element of $\bbT_X(\bC)$ to ensure that the diagram commutes on the nose:
	\begin{equation}
	\label{eqn_cd:archimedean_torsor_compatibility}
		\begin{tikzcd}
			\cU
			\arrow[r, ""]
			\arrow[d, "cl_{\phi}"]
			&
			\gamma_* \cU
			\arrow[r, "c_\gamma"]
			\arrow[d, "cl_{\gamma_*\phi}"]
			& \cU
			\arrow[d, "cl_\phi"]
			\\
			E_M
			\arrow[r, ""]
			&
			\gamma_* E_M
			\arrow[r, ""]
			& E_M
		\end{tikzcd}
	\end{equation}
	It follows that $\tilde{\gamma}_3 \cdot \tilde{\gamma}_2^{-1}\tilde{\gamma}_1^{-1}$ acts on $\cU$ by an automorphism that fixes the fibers of the map $cl_{\phi}$ to $E_M$, and so is an element of the compact subgroup $\bbT_X^c$ if $\bbT_X(\bC)$.
\end{proof}

	\subsection{Examples with obstructions to lifts}
		\label{ssec:Examples-with-obstructions-to-lifts}
		\label{ssec:no_finite_ground_field_lift}
		\label{ssec:calabi_yau_example_without_lifted_action}
		\label{ssec:no_direct_calabi_yau_lift}

We now give two examples showing that the lifting statement for universal torsors cannot be strengthened in two natural ways.
Both examples use the same elementary construction of Calabi--Yau manifolds.
One obstruction is concerned with the need to extend the ground field, while the other is a bit more subtle.

\subsubsection{A construction of Calabi--Yau manifolds}
	\label{sssec:A-construction-of-Calabi-Yau-manifolds}

Let $G$ be a finite group acting on $\bP^1_{\bQ}$.
Assume that $\cO_{\bP^1}(2)$ is $G$-linearized over $\bQ$, and that the induced invariant subsystem of $|\cO_{\bP^1}(4)|$ contains a basepoint-free subsystem defined over $\bQ$.
Fix $n\geq 2$, and set
\[
	Y:=\bP^1\times \bP^{n-1},
	\qquad
	M:=\cO_Y(2,n),
	\qquad
	L_Y:=\cO_Y(1,0),
\]
where $G$ acts on the first factor and trivially on the second.
Then $M$ is $G$-linearized, and the invariant subsystem of
$|M^{\otimes 2}|=|\cO_Y(4,2n)|$
obtained by tensoring with $H^0(\bP^{n-1},\cO(2n))$ is basepoint-free.
By Bertini over the infinite field $\bQ$, choose a smooth $G$-invariant divisor
$D\in |M^{\otimes 2}|$
defined over $\bQ$, and let
\[
	\pi\colon X\to Y
\]
be the double cover branched along $D$, constructed from the square root $M$ of $\cO_Y(D)$.
Since $M$ is $G$-linearized and $D$ is $G$-invariant, the $G$-action on $Y$ lifts to a $G$-action on $X$ over $\bQ$.

The variety $X$ is a smooth projective Calabi--Yau $n$-fold with $\Pic^0(X)=0$.
Indeed,
$K_X=\pi^*(K_Y+M)=\cO_X$
because $K_Y=\cO_Y(-2,-n)$, and
\[
	\pi_*\cO_X=\cO_Y\oplus M^{-1}
\]
gives
\[
	H^i(X,\cO_X)=H^i(Y,\cO_Y)\oplus H^i(Y,\cO_Y(-2,-n)).
\]
The first summand vanishes for $i>0$.
For the second, the Kunneth formula gives
\[
	H^i(Y,\cO_Y(-2,-n))
	\simeq
	\bigoplus_{p+q=i}
	H^p(\bP^1,\cO_{\bP^1}(-2))
	\otimes
	H^q(\bP^{n-1},\cO_{\bP^{n-1}}(-n)).
\]
Here $H^p(\bP^1,\cO(-2))$ is nonzero only for $p=1$, and
$H^q(\bP^{n-1},\cO(-n))$ is nonzero only for $q=n-1$.

\subsubsection{Obstructions to lifts}
	\label{sssec:Obstructions-to-lifts}
For the line bundle
$L:=\pi^*L_Y$
over $X$ we have
\[
	H^0(X,L)=H^0(Y,L_Y)=H^0(\bP^1,\cO_{\bP^1}(1)),
\]
since $H^0(Y,L_Y\otimes M^{-1})=0$.

We shall use this last line as the bridge from linearization obstructions on $\bP^1$ to lifting obstructions for universal torsors.
If a group, or a finite extension of it, acts on a universal torsor of $X$ lifting the action on $X$, then the associated $\bbG_m$-torsor corresponding to the character $[L]\in\Pic(X)$ is equivariant, if the character $[L]$ is $G$-invariant.
Equivalently, $L$ acquires the corresponding linearization.
The induced action on $H^0(X,L)$ is therefore a genuine linear lift of the original projective action on $H^0(\bP^1,\cO_{\bP^1}(1))$.

\begin{proposition}[No finite ground-field lift in general]
	\label{prop:no_finite_ground_field_lift}
	\leavevmode
	\begin{enumerate}
		\item
		There exists a smooth projective Calabi--Yau variety $X$ over $\bQ$ with $\Pic^0(X)=0$ and an action of $G=\bZ/2\bZ$
		on $X$ by $\bQ$-automorphisms such that no finite extension
		\[
			1\to F\to \wtilde{G}\to G\to 1
			\]
			acts on a universal torsor of $X$ by $\bQ$-automorphisms lifting the given $G$-action on $X$.
		\item
		There exists a smooth projective Calabi--Yau variety $X$ over $\bC$ with $\Pic^0(X)=0$ and an action of $G=(\bZ/2\bZ)^2$
		such that the $G$-action on $X$ does not lift to any universal torsor of $X$.
	\end{enumerate}
\end{proposition}
\begin{proof}
	For part (i), let $G=\bZ/2\bZ=\langle a\rangle$ act on $\bP^1_{\bQ}$ by
	\[
		a[u:v]=[2v:u].
	\]
	On $H^0(\bP^1,\cO_{\bP^1}(1))$ the corresponding linear map is represented, up to a scalar in $\bQ^\times$, by
	\[
		A=
		\begin{pmatrix}
			0&2\\
			1&0
		\end{pmatrix},
		\qquad A^2=2I.
	\]
	The line bundle $\cO_{\bP^1}(1)$ admits no finite $G$-extension linearization over $\bQ$.
	Indeed, if $\tilde a$ is any lift of $a$, then its action on $H^0(\bP^1,\cO(1))$ is represented by $qA$ for some $q\in\bQ^\times$.
	Since $\tilde a^2$ lies in the finite kernel, it acts on $\cO(1)$ by a root of unity in $\bQ^\times$, hence by $\pm 1$.
	But $(qA)^2=2q^2I$, and $2q^2=\pm1$ has no solution with $q\in\bQ^\times$.

	On the other hand, $\cO_{\bP^1}(2)$ is $G$-linearized over $\bQ$: the operator $\frac{1}{2}\Sym^2(A)$ squares to the identity.
	The induced invariant sections of $\cO_{\bP^1}(4)$ include
	\[
		u^4+4v^4,\qquad u^3v+2uv^3,\qquad u^2v^2,
	\]
	which have no common zero.
	The construction in \autoref{sssec:A-construction-of-Calabi-Yau-manifolds} therefore produces $X$ and the stated $G$-action.
	If a finite extension $\wtilde{G}$ acted on a universal torsor of $X$ over $\bQ$, then $L$ would admit a finite $G$-extension linearization over $\bQ$.
	The identification of $H^0(X,L)$ in \autoref{sssec:Obstructions-to-lifts} would give such a linearization of $\cO_{\bP^1}(1)$, a contradiction.

	For part (ii), let
	$G=(\bZ/2\bZ)^2$
	act on $\bP^1_{\bC}$ by
	\[
		a[u:v]=[u:-v],
		\qquad
		b[u:v]=[v:u].
	\]
	The standard lifts to $\GL_2$ are
	\[
		A=
		\begin{pmatrix}
			1&0\\
			0&-1
		\end{pmatrix},
		\qquad
		B=
		\begin{pmatrix}
			0&1\\
			1&0
		\end{pmatrix},
	\]
	and they satisfy $AB=-BA$.
	Since any other lifts differ from these by scalars, the action of $G$ on $\bP^1$ does not lift to a genuine linear representation on $H^0(\bP^1,\cO_{\bP^1}(1))$.
	Equivalently, $\cO_{\bP^1}(1)$ is not $G$-linearizable.

	The obstruction is killed after squaring, so $\cO_{\bP^1}(2)$ is $G$-linearized.
	The invariant subsystem of $H^0(\bP^1,\cO(4))$ is spanned by $u^4+v^4$ and $u^2v^2$, hence is basepoint-free.
	The common construction gives a smooth projective Calabi--Yau $X$ over $\bQ$ with $\Pic^0(X)=0$ and a $G$-action.
	If this action lifted to the universal torsor of $X$, then $L$ would be $G$-linearized, and the identification of $H^0(X,L)$ in \autoref{sssec:Obstructions-to-lifts}would give a genuine linear lift of the original projective action on $\bP^1$.
	This contradicts $AB=-BA$.
\end{proof}

\subsubsection{Projective linear virtual lifts}
	\label{ssec:projective_linear_virtual_lifts}
We conclude with a final observation: finitely generated subgroups $\Gamma$ of $\PGL_n(\bC)$ lift to $\GL_n(\bC)$, in fact to $\SL_n(\bC)$, after passing to a finite-index subgroup.
Indeed, we can always lift $\Gamma$ by a finite central extension by roots of unity $\mu_n$ to obtain $\wtilde{\Gamma}\subset \SL_n(\bC)$.
Since $\wtilde{\Gamma}$ is a subgroup of $\SL_n(\bC)$, it is residual finite and hence a finite index subgroup $\Gamma_0$ intersects $\mu_n$ trivially, hence projects injectively to $\Gamma$ and remains of finite index.

One can choose the subgroup $\Gamma$ with the property that it does not preserve a probability measure on $\bP^{n-1}(\bC)$, and neither does any finite-index subgroup of it.
For this, it suffices to choose a Zariski-dense subgroup of $\PGL_n(\bC)$ (viewed as a real algebraic group), e.g. a free group generated by two generic elements.

\subsubsection{An example without Lebesgue-class invariant measure}
	\label{sssec:An-example-without-Lebesgue-class-invariant-measure}

For a group $\Gamma$ generated by a single automorphism, it is of course possible to lift the action to a universal torsor.
We record, however, an example that does not have an invariant measure which is absolutely continuous with respect to Lebesgue.

Specifically, Bedford--Kim \cite{BedfordKim2009_Linear-fractional-recurrences} construct rational surface automorphisms $f$ with positive topological entropy that preserve a cuspidal anticanonical curve, and that moreover are regular on a blowup of $\bP^2$ at finitely many points, denoted $S$ (see \cite[Thm.~2.1]{BedfordKim2009_Linear-fractional-recurrences}).
What suffices for us is that there exists a meromorphic $2$-form $\eta$ on $S$ with a simple pole along the anticanonical curve, and that $f^*\eta=\lambda\eta$ for some $\lambda\in\bC^\times$ with $|\lambda|\neq 1$.
Denote by $\dVol_\eta:=\eta\wedge\overline{\eta}$ the corresponding volume form, of infinite mass.
Suppose, by contradiction, that there exists an $f$-invariant probability measure $\mu$ on $S$ which is absolutely continuous with respect to Lebesgue, hence we can write $\mu = e^\rho \dVol_\eta$ for some Lebesgue-measurable function $\rho$.
The equation $f^*\mu=\mu$ gives $f^*\rho  + 2\log|\lambda| = \rho$, and since $\log|\lambda|\neq 0$ it follows that $\rho$ is strictly monotonic along the orbits of $f$.
This is impossible by Poincar\'e recurrence.







\bibliographystyle{sfilip_bibstyle}
\bibliography{universal_torsor}

\end{document}